\documentclass{amsart}

\usepackage{comment}

\usepackage{amsthm}

\usepackage{enumitem}

\usepackage{amssymb}

\DeclareMathOperator{\ord}{ord}
\DeclareMathOperator{\Gal}{Gal}
\DeclareMathOperator{\Trace}{Trace}
\DeclareMathOperator{\Norm}{Norm}
\DeclareMathOperator{\GL}{GL}

\newcommand{\Q}{\mathbb{Q}}
\newcommand{\Z}{\mathbb{Z}}
\newcommand{\F}{\mathbb{F}}
\newcommand{\OO}{\mathcal{O}}
\newcommand{\fq}{\mathfrak{q}}
\newcommand{\fp}{\mathfrak{p}}

\newtheorem{theorem}{Theorem}
\newtheorem*{thm}{Theorem}
\newtheorem{lemma}{Lemma}[section]

\theoremstyle{definition}

\theoremstyle{remark}

\begin{document}

\title[Local criteria for the unit equation]{Local criteria
for the unit equation\\
 and the asymptotic Fermat's Last Theorem}

\author{Nuno Freitas}

\address{Departament de Matem\`atiques i Inform\`atica,
Universitat de Barcelona (UB),
Gran Via de les Corts Catalanes 585,
08007 Barcelona, Spain}

\email{nunobfreitas@gmail.com}

\author{Alain Kraus}
\address{Sorbonne Universit\'e,
Institut de Math\'ematiques de Jussieu - Paris Rive Gauche,
UMR 7586 CNRS - Paris Diderot,
4 Place Jussieu, 75005 Paris,
France}
\email{alain.kraus@imj-prg.fr}

\author{Samir Siksek}

\address{Mathematics Institute\\
        University of Warwick\\
        CV4 7AL \\
        United Kingdom}

\email{s.siksek@warwick.ac.uk}

\date{\today}
\thanks{Freitas is supported by a Ram\'on y Cajal fellowship with reference RYC-2017-22262.
Siksek is supported by the
EPSRC grant \emph{Moduli of Elliptic curves and Classical Diophantine Problems}
(EP/S031537/1).}

\begin{abstract}
Let $F$ be a totally real number field
of odd degree. We prove several purely local
criteria for the asymptotic Fermat's Last Theorem to hold over $F$,
and also for the non-existence of solutions to the unit equation over $F$.
For example, if $2$ totally ramifies and $3$ splits completely in $F$,
then the asymptotic Fermat's Last Theorem holds
over $F$.
\end{abstract}

\maketitle

\section{Introduction}
Let $F$ be a number field.
The \textbf{asymptotic Fermat's Last Theorem over $F$} is the statement that
there exists a constant~$B_F$, depending only on~$F$, such that, for all
primes~$\ell > B_F$, the only solutions to the equation
$x^\ell + y^\ell + z^\ell = 0$, with $x$, $y$, $z \in F$
satisfy $xyz =0$.  A suitable version of the ABC conjecture \cite{Browkin}
over number fields implies asymptotic FLT for $F$ provided $F$
does not contain a primitive cube root of $1$.
The following two theorems (respectively
\cite[Corollary 1.1]{FKS} and \cite[Theorem 4]{FKS2}) are typical examples
of recent work on asymptotic FLT.
\begin{thm}
Let $F$ be a totally real number field. Suppose
\begin{enumerate}[label=(\roman*)] 
\item $h^+_F$ is odd, where $h^+_F$ denotes the narrow
        class number of $F$;
\item $2$ is totally ramified in $F$.
\end{enumerate}
Then the asymptotic Fermat's Last Theorem holds for $F$.
\end{thm}

\begin{thm}
Let $F$ be a totally real number field and $p \ge 5$
a rational prime. Suppose
\begin{enumerate}[label=(\roman*)] 
\item $F/\Q$ is a
Galois extension of degree $p^m$ for some $m\ge 1$;
\item $p$ is totally ramified in $F$;
\item $2$ is inert in $F$.
\end{enumerate}
Then the asymptotic Fermat's Last Theorem holds for $F$.
\end{thm}
These theorems and others
give asymptotic FLT for
families of number fields subject to restrictions
on the class group,
or on the Galois group.
The purpose of this paper is to establish the following three theorems.
\begin{theorem}\label{thm:pram}
Let $F$ be a totally real number field of degree $n$,
and let $p \ge 5$ be a prime.
Suppose
\begin{enumerate}[label=(\alph*)] 
\item $\gcd(n,p-1)=1$;
\item $2$ is either inert or totally ramifies in $F$;
\item $p$ totally ramifies in $F$.
\end{enumerate}
Then the asymptotic Fermat's Last Theorem holds over $F$.
\end{theorem}
\begin{theorem}\label{thm:23}
Let $F$ be a totally real number field of degree $n$.
Suppose
\begin{enumerate}[label=(\alph*)] 
\item $n\equiv 1$ or $5 \pmod{6}$;
\item $2$ is inert in $F$;
\item $3$ totally splits in $F$.
\end{enumerate}
Then the asymptotic Fermat's Last Theorem holds over $F$.
\end{theorem}
\begin{theorem}\label{thm:23ram}
Let $F$ be a totally real number field of degree $n$.
Suppose
\begin{enumerate}[label=(\alph*)] 
\item $n$ is odd;
\item $2$ totally ramifies in  $F$;
\item $3$ totally splits in $F$.
\end{enumerate}
Then the asymptotic Fermat's Last Theorem holds over $F$.
\end{theorem}
As far as we are aware, Theorems~\ref{thm:pram}, \ref{thm:23} and \ref{thm:23ram}
are the first in the literature giving sufficient criteria
for asymptotic FLT where  the criteria
(apart from restrictions
on the degree) are purely local.

\bigskip

Denote
the ring of integers of $F$ by $\OO_F$,
and the unit group of $\OO_F$ by $\OO_F^\times$.
Associated to $F$ is its
unit equation,
\begin{equation}\label{eqn:unit}
 \lambda + \mu = 1, \quad \lambda, \; \mu \in \OO_F^\times.
\end{equation}
A key step in the proofs of Theorems~\ref{thm:pram}
and \ref{thm:23} is to rule out the existence
of solutions to the unit equation.
For the fields appearing in the statement of Theorem~\ref{thm:pram}
this is furnished by the following theorem.
\begin{theorem}\label{thm:unitcrit}
Let $F$ be a number field of degree $n$, and let $p \ge 5$ be
a prime. Suppose
\begin{enumerate}[label=(\roman*)] 
\item $\gcd(n,(p-1)/2)=1$;
\item $p$ is totally ramified in $F$.
\end{enumerate}
Then the unit equation \eqref{eqn:unit} has no solutions.
\end{theorem}
The following remarkable
theorem of Triantafillou \cite{Triantafillou} rules out solutions to
the unit equation for the fields appearing in the
statement of Theorem~\ref{thm:23}.
\begin{theorem}[Triantafillou]\label{thm:Triantafillou}
Let $F$ be a number field of degree $n$. Suppose
\begin{enumerate}[label=(\roman*)] 
        \item $3 \nmid n$;
        \item $3$ totally splits in $F$.
\end{enumerate}
Then the unit equation \eqref{eqn:unit} has no solutions.
\end{theorem}
Note the similarities in the statements of
Theorems~\ref{thm:unitcrit} and \ref{thm:Triantafillou}:
assumptions (i) in both are restrictions on the
degree, and assumptions (ii) in both are purely local.
Despite the vast literature surrounding unit equations (see \cite{EvertseGyory}
for an extensive survey), the subject of local obstructions to solutions has
received little attention.

We do not expect a common generalization of Theorems~\ref{thm:unitcrit}
and~\ref{thm:Triantafillou}. For example, if we let $K=\Q(\sqrt{-3})$ then the unit equation
has the solution $(\lambda,\mu)=((1+\sqrt{-3})/2,(1-\sqrt{-3})/2)$, showing that Theorem~\ref{thm:unitcrit} is false for $p=3$, and that Theorem~\ref{thm:Triantafillou} is no longer valid if we allow $3$ to ramify instead of splitting.
Another interesting example, given in \cite{FKS2}, is furnished by the number field $K=\Q(\theta)$ with $\theta^3-6\theta^2+9\theta-3$. 
Here $3$ is totally ramified and the unit equation has $18$ solutions including $(\lambda,\mu)=(2-\theta,-1+\theta)$.
As Triantafillou \cite[Remark 3]{Triantafillou} points out, Theorem~\ref{thm:Triantafillou} no longer holds if $3$ is replaced by $5$.

\medskip

We now come to the other main ingredient needed for the proofs
of Theorems~\ref{thm:pram}, \ref{thm:23} and \ref{thm:23ram}.
Suppose $2$ is either inert or totally ramified in $F$,
and write $\fq$ for the unique prime ideal above $2$,
and let $S=\{\fq\}$. We write $\OO_S$ for the ring of $S$-integers
of $F$, and $\OO_S^\times$ for the group of $S$-units.
We consider the $S$-unit equation
\begin{equation}\label{eqn:sunit}
\lambda+\mu=1, \qquad \lambda,~\mu \in \OO_S^\times,
\end{equation}
The following is a special case of
\cite[Theorem 3]{FS1}, which gives a criterion for asymptotic FLT
in terms of the solutions to \eqref{eqn:sunit}.
\begin{theorem}\label{thm:FS}
Let $F$ be a totally real number field.
Assume that $2$ is inert or totally ramified in $F$, write $\fq$ for the unique prime ideal above $2$, and let $S=\{\fq\}$.
If $2$ is totally ramified in $F$, suppose that
        every solution $(\lambda,\mu)$
        to \eqref{eqn:sunit} satisfies
\begin{equation}\label{eqn:condsRam}
\max\{ \lvert \ord_\fq(\lambda) \rvert,~ \lvert \ord_\fq(\mu) \rvert\} \le 4 \ord_\fq(2).
\end{equation}
If $2$ is inert in $F$, suppose that $F$ has odd degree and that
every solution $(\lambda,\mu)$
to \eqref{eqn:sunit} satisfies
\begin{equation}\label{eqn:conds}
\max\{ \lvert \ord_\fq(\lambda) \rvert,~ \lvert \ord_\fq(\mu) \rvert\} \le 4, \qquad
\ord_{\fq}(\lambda \mu)  \equiv 1 \pmod{3}.
\end{equation}
Then the asymptotic Fermat's Last Theorem holds over $F$.
\end{theorem}
The proof \cite{FS1} of this theorem
exploits the strategy of Frey, Serre, Ribet, Wiles and
Taylor, utilized in Wiles' proof \cite{Wiles}
of Fermat's Last Theorem,  and builds on many
deep results
including Merel's uniform boundedness theorem,
and modularity lifting theorems due to
Barnett-Lamb, Breuil, Diamond, Gee, Geraghty, Kisin, Skinner,
Taylor, Wiles, and others.

\medskip

The paper is organized as follows. In Section~\ref{sec:unitcrit}
we prove Theorem~\ref{thm:unitcrit}. In Section~\ref{sec:simplify}
we introduce a lemma that allows us to replace an
arbitrary solution
to the $S$-unit equation \eqref{eqn:sunit} with an integral
solution. In Section~\ref{sec:pram}
we prove Theorem~\ref{thm:pram}. In Section~\ref{sec:Triantafillou}
we give a quick proof of Triantafillou's theorem
(Theorem~\ref{thm:Triantafillou}). This is partly for
the convenience of the reader, but also because
some of the details implicit in Triantafillou's paper
\cite{Triantafillou} are needed for the proofs of Theorems \ref{thm:23} and \ref{thm:23ram}.
In Section~\ref{sec:23} we give proofs of
Theorems~\ref{thm:23} and~\ref{thm:23ram}.
In Section~\ref{sec:Haluk} we give a conjectural
generalization of Theorems~\ref{thm:pram} and \ref{thm:23ram}
to number fields that are not totally real.

\medskip

We are grateful to Alex Bartel for useful discussions, and to the referees for suggesting several improvements.

\section{Proof of Theorem~\ref{thm:unitcrit}}\label{sec:unitcrit}
In this section $F$ is a number field of degree $n$,
and $p$ is a prime that totally ramifies in $F$.
We write $\fp$ for the unique prime of $\OO_F$ above $p$.
\begin{lemma}\label{lem:charpol}
Let $\lambda \in \OO_F$. Write $C_{F,\lambda}(X) \in \Z[X]$ for the characteristic
polynomial of $\lambda$. Then
\begin{equation}\label{eqn:chi}
C_{F,\lambda}(X) \equiv (X-b)^n \mod p \Z[X]
\end{equation}
where $b \in \Z$ satisfies $\lambda \equiv b \pmod{\fp}$.
\end{lemma}
\begin{proof}
Note that $\OO_F/\fp \cong \Z/p\Z$. Hence there is some $b \in \Z$
such that $\lambda \equiv b \pmod{\fp}$.

Let $L$ be the normal closure of $F/\Q$.
Note that $p\OO_F=\fp^n$. Hence $p \OO_L=(\fp \OO_L)^n$.
Let $\sigma \in \Gal(L/\Q)$. Applying $\sigma$ to the previous equality gives
\[
        (\sigma(\fp \OO_L))^n=\sigma(p \OO_L)=p \OO_L=(\fp \OO_L)^n.
\]
        By unique factorization of ideals,  $\sigma(\fp \OO_L)=\fp \OO_L$. Applying $\sigma$ to
$\lambda \equiv b \pmod{\fp \OO_L}$ we find
        $\sigma(\lambda) \equiv b \pmod{\fp \OO_L}$.

Now let $\lambda_1,\dotsc,\lambda_n$ be the roots in $L$
of the characteristic polynomial $C_{F,\lambda}(X)$.
Since $C_{F,\lambda}$ is a power of
the minimal polynomial of $\lambda$, the $\lambda_i$ are all
conjugates of~$\lambda$.
Therefore $\lambda_i \equiv b \pmod{\fp \OO_L}$ for all~$i$.
Thus
\[
C_{F,\lambda}(X)=(X-\lambda_1)(X-\lambda_2)\cdots (X-\lambda_n)
\equiv (X-b)^n \mod \fp \OO_L[X].
\]
However $C_{F,\lambda}(X)$ and $(X-b)^n$ both belong to $\Z[X]$.
This gives \eqref{eqn:chi}.
\end{proof}

\begin{lemma}\label{lem:norm}
Let $\lambda \in \OO_F$, and let $b \in \Z$ satisfy
$\lambda \equiv b \pmod{\fp}$.
Then
\[
        \Norm_{F/\Q}(\lambda) \equiv b^n \pmod{p}.
\]
\end{lemma}
\begin{proof}
We obtain this immediately on comparing constant
        coefficients in \eqref{eqn:chi}.
\end{proof}

\begin{lemma}\label{lem:pm1}
Suppose $p$ is odd, and that $\gcd(n,(p-1)/2)=1$.
Let $\lambda \in \OO_F^\times$.
Then $\lambda \equiv \pm 1 \pmod{\fp}$.
\end{lemma}
\begin{proof}
Our assumption $\gcd(n,(p-1)/2)=1$ is equivalent  to
the existence of integers
$u$, $v$ so that $un+v(p-1)/2=1$.
Let $b \in \Z$
satisfy $\lambda \equiv b \pmod{\fp}$. By Lemma~\ref{lem:norm}
and the fact that $\lambda$ is a unit,
$b^n \equiv \pm 1 \pmod{p}$.
However, $b^{(p-1)/2} \equiv \pm 1 \pmod{p}$.
It follows that
\[
        b\; \equiv \; 
        (b^n)^u \cdot \left(b^{(p-1)/2} \right)^v 
        \; \equiv \;
        \pm 1 \pmod{p}. 
\]
Therefore $\lambda \equiv \pm 1 \pmod{\fp}$.
\end{proof}

\begin{proof}[Proof of Theorem~\ref{thm:unitcrit}]
Let $F$ be a number field of degree $n$ and let $p \ge 5$
be a prime that totally ramifies in $F$, and such that
$\gcd(n,(p-1)/2)=1$.
        Let $(\lambda,\mu)$ be a solution to the unit
        equation \eqref{eqn:unit}.
        By Lemma~\ref{lem:pm1}, we see that $\lambda \equiv \pm 1 \pmod{\fp}$
        and $\mu \equiv \pm 1 \pmod{\fp}$. Thus $1=\lambda+\mu \equiv \pm 1 \pm 1 \pmod{\fp}$. As $p \ne 3$, this is impossible.
\end{proof}

\section{A simplifying lemma}\label{sec:simplify}
Let $F$ be a number field in which $2$ is either inert or totally ramified. We
write $\fq$ for the unique prime above $2$ and let $S=\{\fq\}$. To deduce Theorems \ref{thm:pram},
\ref{thm:23} and \ref{thm:23ram} from Theorem~\ref{thm:FS} we need strong control
of the solutions to the $S$-unit equation \eqref{eqn:sunit}.
The following lemma facilitates this  by allowing
us replace arbitrary solutions by
integral ones.
\begin{lemma}\label{lem:simplify}
Let $(\lambda,\mu)$ be a solution to the $S$-unit equation \eqref{eqn:sunit}, and write
\begin{equation}\label{eqn:nlmdef}
m_{\lambda,\mu}=\max\{ \lvert \ord_\fq(\lambda) \rvert,~ \lvert \ord_\fq(\mu) \rvert\}.
\end{equation}
Then there is a solution $(\lambda^\prime,\mu^\prime)$ to \eqref{eqn:sunit}
with
\[
\lambda^\prime \in \OO_F \cap \OO_S^\times, \qquad \mu^\prime \in \OO_F^\times, \qquad
m_{\lambda^\prime,\mu^\prime}=m_{\lambda,\mu}. 
\]
\end{lemma}

\begin{proof}
Let $(\lambda,\mu)$ be a solution to \eqref{eqn:sunit}.
If $\ord_\fq(\lambda)=\ord_\fq(\mu)=0$ then
$\lambda$, $\mu \in \OO_F^\times$ and we take
$\lambda^\prime=\lambda$ and $\mu^\prime=\mu$.
If $\ord_\fq(\lambda)>0$ then the relation $\lambda+\mu=1$
forces $\ord_\fp(\mu)=0$.
In this case we have
$\lambda \in \OO_F$, $\mu \in \OO_F^\times$ and
we again take $\lambda^\prime=\lambda$ and $\mu^\prime=\mu$.
If $\ord_\fq(\mu)>0$ then
we take $\lambda^\prime=\mu$ and $\mu^\prime=\lambda$.
We have therefore reduced to the case where
$\ord_\fq(\lambda)<0$ and $\ord_\fq(\mu)<0$.
From the relation $\lambda+\mu=1$ we have
$\ord_\fq(\lambda)=\ord_\fq(\mu)=-t$
for some positive $t=m_{\lambda,\mu}$.
In this case the lemma follows on choosing
 $\lambda^\prime=1/\lambda$, $\mu^\prime=-\mu/\lambda$.
\end{proof}

\section{Proof of Theorem~\ref{thm:pram}}\label{sec:pram}
In this section we suppose that $F$ and $p$
satisfy the hypotheses of Theorem~\ref{thm:pram}.
Namely, $F$ is a totally real field of degree $n$
such that $2$ is either inert or totally ramifies in $F$,
and $p \ge 5$ is a prime totally ramified in $F$
and satisfying $\gcd(n,p-1)=1$.
As before we take
$\fq$ to be the unique prime above $2$ and $S=\{\fq\}$, and we write
$\fp$ for the unique prime above $p$.

\begin{lemma}\label{lem:valbound}
Every solution $(\lambda,\mu)$
to the $S$-unit equation \eqref{eqn:sunit} satisfies
\[
        m_{\lambda,\mu} \, < \, 2 \ord_\fq(2),
\]
where $m_{\lambda,\mu}$ is defined in \eqref{eqn:nlmdef}.
\end{lemma}

\begin{proof}
Suppose that $m_{\lambda,\mu} \ge 2 \ord_\fq(2)$.
By Lemma~\ref{lem:simplify},
there is a solution $(\lambda^\prime,\mu^\prime)$
to the $S$-unit equation \eqref{eqn:sunit}
with $\lambda^\prime \in \OO_F \cap \OO_S^\times$ and $\mu^\prime \in \OO_F^\times$
so that
$\ord_\fq(\lambda^\prime) 
= m_{\lambda^\prime,\mu^\prime} = m_{\lambda,\mu} \ge 2\ord_\fq(2)$.
Since $\mu^\prime=1-\lambda^\prime$ we see that
$\mu^\prime \equiv 1 \pmod{4}$.
Hence $\Norm_{F/\Q}(\mu^\prime) \equiv 1 \pmod{4}$.
But $\Norm_{F/\Q}(\mu^\prime)=\pm 1$ as $\mu^\prime$ is a unit.
Hence $\Norm_{F/\Q}(\mu^\prime)=1$.

Next we utilize the assumption $\gcd(n,p-1)=1$.
From Lemma~\ref{lem:pm1}, we have
$\mu^\prime \equiv \pm 1 \pmod{\fp}$.
If $\mu^\prime \equiv 1 \pmod{\fp}$
then $\lambda^\prime =1-\mu^\prime\equiv 0 \pmod{\fp}$
and this gives a contradiction, since $\lambda^\prime \in \OO_S^\times$
and $\fp \notin S$. Thus $\mu^\prime \equiv -1 \pmod{\fp}$.
But as $\gcd(n,p-1)=1$, the degree $n$ is odd.
By Lemma~\ref{lem:norm} we have
$\Norm_{F/\Q}(\mu^\prime) \equiv (-1)^n \equiv -1 \pmod{p}$
and so $\Norm_{F/\Q}(\mu^\prime)=-1$. This gives a contradiction,
thereby establishing the lemma.
\end{proof}

Note that the $S$-unit equation \eqref{eqn:sunit}
has the following three solutions $(\lambda,\mu)=(1/2,1/2)$, $(-1,2)$, $(2,-1)$.
The following lemma says that if $2$ is inert then every other
solution must
have the same valuations as these.
\begin{lemma}\label{lem:solvals}
Suppose $2$ is inert in $F$.
Then every solution
$(\lambda,\mu)$ to the $S$-unit equation \eqref{eqn:sunit}
satisfies
\[
        (\ord_\fq(\lambda),\ord_\fq(\mu))  \in \{ (-1,-1),~(0,1),~(1,0)\}.
\]
\end{lemma}
\begin{proof}
Let $(\lambda,\mu)$ be a solution to \eqref{eqn:sunit}.
        From Lemma~\ref{lem:valbound} we know that $m_{\lambda,\mu}=0$ or $1$.
However if $m_{\lambda,\mu}=0$ then $(\lambda,\mu)$
is a solution to the unit equation \eqref{eqn:unit}
and this contradicts Theorem~\ref{thm:unitcrit}.
Therefore $m_{\lambda,\mu}=1$. Now this combined
with the relation $\lambda+\mu=1$ yields the lemma.
\end{proof}

\begin{proof}[Proof of Theorem~\ref{thm:pram}]
If $2$ is ramified then Lemma~\ref{lem:valbound}
        ensures that every solution to \eqref{eqn:sunit}
        satisfies \eqref{eqn:condsRam}.
If $2$ is inert the Lemma~\ref{lem:solvals} assures us that every solution
to \eqref{eqn:sunit} satisfies \eqref{eqn:conds}.
Moreover, since $\gcd(n,p-1)=1$, the degree of $F$ is odd.
Theorem~\ref{thm:pram} follows immediately from
Theorem~\ref{thm:FS}.
\end{proof}

\section{Proof of Theorem~\ref{thm:Triantafillou}}\label{sec:Triantafillou}
The following lemma is a slight
generalization of an idea that is implicit in \cite{Triantafillou}.
\begin{lemma}\label{lem:3adic}
Let $F$ be a number field in which $3$ splits completely.
Let $S$ be a finite set of primes of $F$ that is disjoint
from the primes above $3$.
Let $(\lambda,\mu)$
be a solution to the $S$-unit equation \eqref{eqn:sunit}
with $\lambda$, $\mu \in \OO_F$.
Then
\[
\lambda \equiv \mu \equiv -1 \pmod{3}.
\]
\end{lemma}
\begin{proof}
Let $\fp_1,\dotsc,\fp_n$ be the primes of $F$
above $3$. Then $\OO_F/\fp_i\cong \F_3$, and so the possible
residue classes modulo $\fp_i$ are $0$, $1$, $-1$.
However $\lambda$, $\mu \in \OO_S^\times$ and $\fp_i \notin S$, so
$\lambda \not\equiv 0 \pmod{\fp_i}$
and $\mu \not\equiv 0 \pmod{\fp_i}$. Hence
$\lambda \equiv \pm 1 \pmod{\fp_i}$ and $\mu \equiv \pm 1 \pmod{\fp_i}$.
But $\lambda+\mu=1$. It follows that
$\lambda \equiv \mu \equiv -1 \pmod{\fp_i}$.
The lemma follows as
$3\OO_F$ is the product of the distinct primes $\fp_1,\dotsc,\fp_n$.
\end{proof}

\begin{proof}[Proof of Theorem~\ref{thm:Triantafillou}]
Let $F$ be a number field of degree $n$. Suppose
that $3 \nmid n$
and that $3$ splits completely in $F$.
Let $(\lambda,\mu)$ be a solution to the unit
equation \eqref{eqn:unit}. By Lemma~\ref{lem:3adic},
applied with $S=\emptyset$, we have
$\lambda=-1+3\phi$, $\mu=-1+3\psi$
where $\phi$, $\psi \in \OO_F$. Moreover, from $\lambda+\mu=1$
we obtain $\phi+\psi=1$.
Let
$\phi_1,\dotsc,\phi_n$ be the images of $\phi$ under the
$n$ embeddings $F \hookrightarrow \overline{F}$. As $\lambda$
is a unit
\begin{multline*}
\pm 1=\Norm_{F/\Q}(\lambda)=(-1+3\phi_1) \cdots (-1+3\phi_n)
\equiv\\ (-1)^n + (-1)^{n-1} \cdot 3 \Trace_{F/\Q}(\phi) \pmod{9}.
\end{multline*}
By considering all the choices for $\pm 1$ and $(-1)^n$,
we obtain
$3 \Trace_{F/\Q}(\phi) \equiv -2$, $2$ or $0
\pmod{9}$. The first two are plainly  impossible and so
$\Trace_{F/\Q}(\phi) \equiv 0 \pmod{3}$. Similarly $\Trace_{F/\Q}(\psi)
\equiv 0 \pmod{3}$.
But, as $\phi+\psi=1$,
\[
        n=\Trace_{F/\Q}(\phi+\psi)=\Trace_{F/\Q}(\phi)+\Trace_{F/\Q}(\psi) \equiv 0 \pmod{3},
\]
giving a contradiction.
\end{proof}

\section{Proof of Theorems~\ref{thm:23} and~\ref{thm:23ram}}\label{sec:23}
\subsection*{Proof of Theorem~\ref{thm:23}}
We now prove Theorem~\ref{thm:23}. Thus we let
$F$ be a totally real field of degree $n \equiv 1$ or $5 \pmod{6}$,
and suppose that $2$ is inert in $F$ and $3$ totally splits in $F$.
As before, we write $\fq=2\OO_F$,
and let $S=\{\fq\}$. Note that $2$ and $3$
do not divide the degree $n$.
To deduce Theorem~\ref{thm:23} from Theorem~\ref{thm:FS}
all we need to do is show that every solution $(\lambda,\mu)$
to the $S$-unit equation \eqref{eqn:sunit}
satisfies \eqref{eqn:conds}.
Just as in the proof of Theorem~\ref{thm:pram}
it is enough to show that $m_{\lambda,\mu}=1$
for every solution $(\lambda,\mu)$ to \eqref{eqn:sunit}.
We know from Theorem~\ref{thm:Triantafillou}
that $m_{\lambda,\mu} \ne 0$.
Suppose $m_{\lambda,\mu} \ge 2$.
By Lemma~\ref{lem:simplify}
there is a solution $(\lambda^\prime,\mu^\prime)$
to \eqref{eqn:sunit}
such that $\lambda^\prime \in \OO_F$,
$\mu^\prime \in \OO_F^\times$,
and $\ord_\fq(\lambda^\prime)=m_{\lambda^\prime,\mu^\prime}=
m_{\lambda,\mu} \ge 2$. Thus $\mu^\prime=1-\lambda^\prime \equiv 1 \pmod{4}$,
and hence $\Norm_{F/\Q}(\mu^\prime)=1$.

However, by Lemma~\ref{lem:3adic},
we have $\mu^\prime \equiv -1 \pmod{3}$
and so $\Norm_{F/\Q}(\mu^\prime)=(-1)^n=-1$
since $n$ is odd.
This gives a contradiction, and completes the proof of
Theorem~\ref{thm:23}.

\subsection*{Proof of Theorem~\ref{thm:23ram}}
Finally we prove Theorem~\ref{thm:23ram}. Thus we let
$F$ be a totally real field of odd degree $n$,
and suppose that $2$ is totally ramified in $F$ and that $3$ totally splits in $F$.
We write $\fq$ for the unique prime above $2$
and let $S=\{\fq\}$.
We claim that every solution to $(\lambda,\mu)$ to the
$S$-unit equation \eqref{eqn:sunit} satisfies $m_{\lambda,\mu} < 2 \ord_{\fq}(2)$.
Our claim combined with Theorem~\ref{thm:FS} immediately implies
Theorem~\ref{thm:23ram}.
Suppose $(\lambda,\mu)$ is a solution to the $S$-unit equation
with $m_{\lambda,\mu} \ge 2 \ord_{\fq}(2)$.
By Lemma~\ref{lem:simplify}
there is a solution $(\lambda^\prime,\mu^\prime)$
to \eqref{eqn:sunit}
such that $\lambda^\prime \in \OO_F$,
$\mu^\prime \in \OO_F^\times$,
and $\ord_\fq(\lambda^\prime)=m_{\lambda^\prime,\mu^\prime}=
m_{\lambda,\mu} \ge 2\ord_\fq(2)$. Thus $\mu^\prime=1-\lambda^\prime \equiv 1 \pmod{4}$.
The remainder of the argument is identical to that in the above proof of Theorem~\ref{thm:23}.

\section{A conjectural generalization to arbitrary number fields}\label{sec:Haluk}
Theorem~\ref{thm:FS} is critical for the proofs of Theorems~\ref{thm:pram},
\ref{thm:23} and \ref{thm:23ram}.
As previously mentioned, the proof of that theorem relies
on the extraordinary progress in proving modularity lifting
theorems over totally real fields.
Unfortunately, our understanding of modularity over non-totally real
number fields is largely conjectural. However, in \cite{Haluk},
a version of Theorem~\ref{thm:FS} is established
for general (as opposed to totally real)  number fields assuming
two standard conjectures from the Langlands programme.
In this section we give versions of Theorems~\ref{thm:pram}
and \ref{thm:23ram} which are valid for general number
fields $F$, assuming those two conjectures. For the precise
statements of the two conjectures, we refer to \cite{Haluk};
instead we give a brief indication of what they are:
\begin{itemize}
\item Conjecture 3.1 of \cite{Haluk} is a weak version
of Serre's modularity conjecture
for odd, absolutely irreducible, continuous $2$-dimensional
mod $\ell$ representations of $\Gal(\overline{F}/F)$
that are finite flat at every prime above $\ell$;
\item Conjecture 4.1 of \cite{Haluk} states that
every weight $2$ newform for $\GL_2$ over $F$
with integer Hecke eigenvalues has an associated
elliptic curve over $F$, or a fake elliptic curve
over $F$.
\end{itemize}
The following is a special case of \cite[Theorem 1.1]{Haluk}.
\begin{theorem}\label{thm:Haluk} (\c{S}eng\"{u}n and Siksek)
Let $F$ be a number field for which
Conjectures 3.1 and 4.1 of \cite{Haluk} hold.
Assume that $2$ is totally ramified in $F$, write $\fq$ for the unique prime ideal above $2$, and let $S=\{\fq\}$.
Suppose that
        every solution $(\lambda,\mu)$
        to \eqref{eqn:sunit} satisfies
\[
\max\{ \lvert \ord_\fq(\lambda) \rvert,~ \lvert \ord_\fq(\mu) \rvert\} \le 4 \ord_\fq(2).
\]
Then the asymptotic Fermat's Last Theorem holds over $F$.
\end{theorem}

We note that the assumption that $F$ is totally real played no role in the proofs
of Theorems \ref{thm:pram}, \ref{thm:23} and \ref{thm:23ram} except when
invoking Theorem \ref{thm:FS}. We also note that Theorems \ref{thm:FS}
and \ref{thm:Haluk} have identical hypotheses and conclusions
for the case when $2$ is totally ramified in $F$, except
for the two additional conjectural assumptions in Theorem \ref{thm:Haluk}.
The following two theorems are proved simply by
invoking Theorem \ref{thm:Haluk} instead of Theorem \ref{thm:FS}
in the proofs of
Theorems~\ref{thm:pram} and~\ref{thm:23ram}.
\begin{theorem}\label{thm:pramGen}
Let $F$ be a number field of degree $n$,
for which
Conjectures 3.1 and 4.1 of \cite{Haluk} hold,
and let $p \ge 5$ be a prime.
Suppose
\begin{enumerate}[label=(\alph*)]
\item $\gcd(n,p-1)=1$;
\item $2$ totally ramifies in $F$;
\item $p$ totally ramifies in $F$.
\end{enumerate}
Then the asymptotic Fermat's Last Theorem holds over $F$.
\end{theorem}
\begin{theorem}\label{thm:23ramGen}
Let $F$ be a number field of degree $n$,
for which
Conjectures 3.1 and 4.1 of \cite{Haluk} hold.
Suppose
\begin{enumerate}[label=(\alph*)]
\item $n$ is odd;
\item $2$ totally ramifies in  $F$;
\item $3$ totally splits in $F$.
\end{enumerate}
Then the asymptotic Fermat's Last Theorem holds over $F$.
\end{theorem}
Unfortunately, we are unable to prove similar statements in the case $2$ is inert.
Indeed, the existence of a degree $1$ prime above $2$ is critical
in \cite{Haluk} at two points. It is needed when proving that the
mod $\ell$ representation of the Frey elliptic curve is absolutely irreducible,
which is a prerequisite for applying \cite[Conjecture 3.1]{Haluk}.
It is also needed after invoking \cite[Conjecture 4.1]{Haluk}
to rule out the possibility that a particular weight $2$
newform with rational Hecke eigenvalues is associated
to a fake elliptic curve. We note in passing that the Fermat equation $x^\ell+y^\ell+z^\ell=0$ has the solution $(1,\zeta_3,\zeta_3^2)$ for all $\ell \ne 3$, where $\zeta_3=(-1+\sqrt{-3})/2$.
The existence of this solution suggests that variants of the above theorems with $2$ inert might be harder to prove.

\bibliographystyle{plain}
\bibliography{FLT.bib}

\end{document}